\documentclass{amsart}
\usepackage{amsmath,amssymb}
\usepackage{graphicx,subfigure}

\usepackage{color}

\newtheorem{theorem}{Theorem}[section]
\newtheorem{proposition}[theorem]{Proposition}

\newtheorem{lemma}[theorem]{Lemma}

\theoremstyle{definition}
 \newtheorem{definition}[theorem]{Definition}

\newtheorem*{definition*}{Definition}

\theoremstyle{remark}
\newtheorem{remark}[theorem]{Remark}

\numberwithin{equation}{section}

\newcommand{\al}{\alpha}

\newcommand{\la}{\lambda}

\newcommand{\vp}{\varphi}

\newcommand{\De}{\Delta}
\newcommand{\Ga}{\Gamma}

\newcommand{\Om}{\Omega}

\def\NN{\mathbb{N}}
\def\RR{\mathbb{R}}

\def\ZZ{\mathbb{Z}}

\newcommand{\cD}{{\mathcal D}}

\newcommand{\cH}{{\mathcal H}}

\newcommand{\cK}{{\mathcal K}}

\newcommand{\pd}{\partial}
\newcommand\minus\backslash

\newcommand\lan\langle
\newcommand\ran\rangle

\DeclareMathOperator\Div{div}

\DeclareMathOperator\Vol{Vol} 
 
\def\Harm{\cH_\Om }

\renewcommand\leq\leqslant
\renewcommand\geq\geqslant
\newlength{\intwidth}

\newcommand\BOm{\overline\Om}

 \DeclareMathOperator\curl{curl}

\newcommand\restr{\!\upharpoonright_{\pd\Om}}\newcommand\restrt{\!\upharpoonright_{\pd\Om^t}}

\newcommand\Dom{\cD_\Om}

\begin{document}

\title[Optimal convex domains for curl]{Optimal convex domains\\ for the first curl eigenvalue}

\author{Alberto Enciso}
\address{Instituto de Ciencias Matem\'aticas, Consejo Superior de
  Investigaciones Cient\'\i ficas, 28049 Madrid, Spain}
\email{aenciso@icmat.es, wadim.gerner@icmat.es, dperalta@icmat.es}

\author{Wadim Gerner}

\author{Daniel Peralta-Salas}

%
%
\begin{abstract}
We prove that there exists a bounded convex domain $\Omega\subset \mathbb{R}^3$ of fixed volume that minimizes the first positive curl eigenvalue among all other bounded convex domains of the same volume. We show that this optimal domain cannot be analytic, and that it cannot be stably convex if it is sufficiently smooth (e.g., of class $C^{1,1}$). Existence results for uniformly H\"older optimal domains in a box (that is, contained in a fixed bounded domain $D\subset \RR^3$) are also presented.
\end{abstract}
\maketitle

\section{Introduction}

Let $\Om\subset\RR^3$ be a bounded domain. The eigenfunctions of curl on~$\Om$ are divergence-free vector fields that satisfy
\begin{align}\label{P}
\curl u_k=\mu_k(\Om)\, u_k\quad\text{in }\Om\,, \qquad u_k\restr\cdot N=0
\end{align}
and are orthogonal in~$L^2$ to the space of harmonic fields $\cH_\Om$ on~$\Om$
that are tangent to the boundary (for a precise definition see Section~\ref{S.Lips}); here $N$ is a unit normal field on $\pd\Om$. The curl eigenfunctions are also known as Beltrami fields or force-free fields in different contexts, and it is well known that they play a preponderant role in the analysis of physical systems described by solenoidal vector fields, such as fluid mechanics, electromagnetism or magnetohydrodynamics.

In this paper we are mainly interested in convex domains. Without any a priori regularity assumption, it is standard that a bounded convex domain is homeomorphic to a ball and has a Lipschitz boundary, cf.~\cite[Lemma 2.3]{DL04}. Accordingly, we need to make sense of the spectral problem~\eqref{P} for Lipschitz-continuous domains.

More precisely, in
Section~\ref{S.sa} we will recall that curl defines a self-adjoint
operator with compact resolvent whose domain, which we will denote by $\Dom$, is dense in the space
\[
\cK(\Om)=\Big\{ v\in L^2(\Om): \int_\Om v\cdot w\, dx=0\;
\text{for all } w\in L^2(\Om) \text{ with } \curl w=0\Big\}\,,
\]
and $u_k$ is assumed to belong to the domain of this operator. When
the domain $\Om$ is smooth, the self-adjointness of curl was established by
Giga and Yoshida~\cite{Giga}; see also~\cite{Hiptmair} for a recent
analysis of the self-adjoint extensions of curl.

This self-adjoint operator, which we still denote by $\curl$, has infinitely many positive and negative
eigenvalues $\{\mu_k(\Om)\}_{k=-\infty}^\infty$ with finite multiplicity, which tend to
$\pm\infty$  as
$k\to\pm\infty$ and which one can label so that
\[
\cdots \leq \mu_{-3}(\Om)\leq \mu_{-2}(\Om)\leq \mu_{-1}(\Om) <0<\mu_1(\Om)\leq \mu_2(\Om)\leq \mu_3(\Om)\leq \cdots
\]
We will refer to~$\mu_1(\Om)$ and $\mu_{-1}(\Om)$ as the {\em first
  positive eigenvalue}\/ and {\em first negative eigenvalue}\/ of the curl
operator, respectively. Note that their multiplicity can be higher than~1 (as is the case when $\Om$ is a spherically symmetric domain~\cite{Canta}).

When $\Om$ is smooth, one can equivalently write
\[
\cK(\Om)=\Big\{ v\in L^2(\Om): \Div v=0\,,\; v\restr\cdot
N=0\,,\; \int_\Om
v\cdot h\, dx=0 \; \text{for all } h\in\Harm \Big\}\,,
\]
where $\Harm $ is the space of harmonic fields on~$\Om$
that are tangent to the boundary (see Section~\ref{S.Lips} for a precise definition). It is standard that $\text{dim}(\Harm)=g(\Om)$, where $g(\Om)$ is the genus of $\pd\Om$. In the case of Lipschitz domains $N\cdot v\restr$ cannot be
defined as an element of $H^{-\frac12}(\pd\Om)$ for all $v\in
L^2(\Om)$ (see e.g.~\cite{Buffa}).

The main goal of this article is to explore the existence of domains that minimize the first (positive or negative) curl eigenvalue among domains with the same volume. This optimization problem is not only interesting from a spectral theoretic view point (as a natural vectorial analogue of the corresponding classical problem for the Dirichlet Laplacian), but it is also motivated by Woltjer's variational principle~\cite{Wolt}. This principle establishes that the eigenfunctions associated with the first curl eigenvalue minimize the $L^2$ norm among all the divergence-free fields on $\Om$ with fixed helicity and tangent to the boundary (the so called, Taylor states), so they are natural candidates to be relaxed states of ideal plasmas.

Despite its importance, the existence of these optimal domains (of class $C^{2,\alpha}$) has been addressed only recently in~\cite{EP21}; see also~\cite[Chapter 2]{GernerThesis} for a study of the optimal domain problem in the context of Riemannian $3$-manifolds. It was proved that such smooth optimal domains, if they exist, cannot be homeomorphic to a ball, and in the case that they are axisymmetric, they cannot have a convex section. This is in strong contrast with the case of the Dirichlet Laplacian, where the famous Faber--Krahn inequality implies that the ball is the only optimal domain for the first eigenvalue. Contrary to the case of higher eigenvalues of the Dirichlet Laplacian~\cite{BDM,Henrot}, the existence of optimal shapes for the first positive curl eigenvalue, even in the class of quasi-open sets, is unclear.

Specifically, in this work we are interested in the existence of optimal domains for the curl operator within the class of convex sets (without any a priori regularity assumption). To this end, we introduce the following definition:

\begin{definition}
A bounded convex domain $\Om\subset\RR^3$ is {\em optimal}\/ for the first
positive
curl eigenvalue if
\[
\mu_1(\Om)\leq \mu_1(\Om')
\]
for any convex domain~$\Om'$ of the same volume.
\end{definition}

As we are considering the curl operator on the whole space~$\RR^3$, the scaling properties of the eigenvalue equation ensure that the concrete volume of the domain~$\Om$ is irrelevant. That is, $\Om$ is an optimal domain of volume~$V$ if and only if the rescaled domain $\la\Om$ is optimal with volume~$\la^3V$. This will not be the case when we consider optimal domains in a box later on. Also, as
the curl eigenvalues of the reflected domain $-\Om:=\{-x:x\in \Om\}$ satisfy the identity $\mu_k(-\Om)=-\mu_{-k}(\Om)$, one should note that the results that we shall prove about optimal domains for the first positive curl eigenvalue trivially extend to the case of the first negative curl eigenvalue.

Our main theorem shows that there exist optimal convex domains, and that they are not analytic. This is consistent with the results in~\cite{EP21} in the sense that the optimal domains do not need to be smooth, and moreover, they do not need to be stable in the sense that convexity may be lost under arbitrarily small volume preserving deformations. This is made precise in Proposition~\ref{MT3}, cf.~Section~\ref{analytic}, showing that if the optimal convex domain is regular enough it cannot be stably convex.

\begin{theorem}\label{T1}
There exists a bounded convex domain $\Omega$ of any fixed volume which is optimal for the first positive curl eigenvalue. Optimal convex domains are not analytic and, if they are~$C^{1,1}$, they are not stably convex.
\end{theorem}

\begin{remark}
An analogous existence result holds for the class of axially symmetric convex domains.
\end{remark}

The proof of Theorem~\ref{T1} is presented in Sections~\ref{proof} and~\ref{analytic}. Although we follow the same strategy as in the case of the Dirichlet Laplacian, several nontrivial technical difficulties  arise when trying to adapt the argument to the curl operator. First, using certain monotonicity and scaling properties we show the continuity of the curl eigenvalues with respect to the Hausdorff distance between compact sets. The most challenging part of the proof is then to establish that for any sequence of bounded convex domains of fixed volume, the first curl eigenvalue grows unless the diameters of the sets of the sequence are uniformly bounded. This requires a new estimate for the first curl eigenvalue in cylindrical domains. Finally, to prove that the domain cannot be analytic, we make use of an integral identity on the boundary of an optimal domain that involves the pointwise norm $|u_1|\restr$. The existence of optimal domains for the class of convex and axially symmetric sets is established in Section~\ref{SS.axi}.

Concerning the regularity problem, we do not know if an optimal convex domain is necessarily $C^1$, as is the case for the Dirichlet Laplacian~\cite{Bucur}. The proof in~\cite{Bucur} exploits the notion of $\Gamma$-convergence, which is well adapted to the Laplacian. Although it is possible to define a suitable notion of $\Gamma$-convergence for the curl operator, which in fact enjoys nice properties such as Lipschitz continuity of the curl eigenvalues, the vectorial nature and the non-scalar boundary conditions of problem~\eqref{P} makes it hard to exploit the aforementioned $\Gamma$-convergence to study the regularity of the optimal domains. See Appendix~\ref{app} for details.

The second main result that we prove concerns the existence of optimal domains in a box for the class of uniformly H\"older sets. This result is easier than Theorem~\ref{T1} as it concerns sequences of domains that are uniformly bounded, as they are assumed to be contained in a fixed smooth bounded domain $D\subset\RR^3$. In the statement we use the family $\mathcal{D}^{k,\alpha}(L,c_0,c_1)$ formed by domains $\Om\subset D$ of class $C^{k,\alpha}$ with uniform H\"older and gradient bounds given by $(L,c_0,c_1)$. From now on we assume that the parameters $(L,c_0,c_1)$ are {\em admissible}\/ in the sense that there are domains of volume~$V$ in $\mathcal{D}^{k,\alpha}(L,c_0,c_1)$. This is not problematic, since any $C^{k,\alpha}$ domain $\Om\subset D$ obviously belongs to $\mathcal{D}^{k,\alpha}(L,c_0,c_1)$ if $L,c_0^{-1},c_1^{-1}$ are large enough. Precise definitions are given in Section~\ref{SS.Holder}.

\begin{theorem} \label{T.main2}
For any $V\in(0,|D|)$,  $k\geq2$,  $\al\in(0,1]$ and any admissible constants $L,c_0,c_1$, there exists a domain $\Omega\subset D$ of volume~$V$ that is optimal for the first positive curl eigenvalue within the class $\mathcal{D}^{k,\alpha}(L,c_0,c_1)$.
\end{theorem}

\section{Summary of the spectral theory of curl on Lipschitz and convex domains}
\label{S.sa}

In this section, following previous literature, we introduce the suitable functional-analytic setting for the self-adjointness of the curl operator on bounded Lipschitz domains, and establish some regularity properties of its eigenfunctions (Section~\ref{S.Lips}). This general result for Lipschitz domains (which is of interest in itself) is applied to the case of convex domains in Section~\ref{S.convpre}, where we also show the stability under close to the identity deformations of strongly convex domains.

First, we recall that a domain $\Omega\subset\RR^3$ is \textit{convex} if for any $x,y\in \overline\Omega$ we have $\lambda x+(1-\lambda)y\in \overline\Omega$ for all $0\leq \lambda\leq 1$. If $\Omega$ is a bounded domain with $C^1$ boundary we say that $\Omega$ is \textit{strongly convex} if
\[
\Om=\{x\in\mathbb R^3: f(x)< 0\}\,, \qquad \RR^3\backslash\overline{\Om}=\{x\in\mathbb R^3: f(x)> 0\}
\]
for some strongly convex function $f\in C^1(\RR^3)$ with $\nabla f(x)\neq 0$ for all $x\in\pd\Om$. $f$ is called a \textit{defining function}~\cite{Krantz} for $\Om$. We recall that $f$ is strongly convex if
\begin{equation}
(\nabla f(x)-\nabla f(y))\cdot (x-y)\geq \sigma|x-y|^2
\end{equation}
for some constant $\sigma>0$ and all $x,y\in \RR^3$; see~\cite{Vial} for different equivalent definitions of strong convexity. In particular, in the case that $\Om$ is $C^2$, it is well known that it is strongly convex if and only if its Gauss curvature is positive on $\partial\Omega$. Finally, we say that a bounded convex domain $\Om$ is \emph{stable} if for any smooth, compactly supported vector field $v$ on $\mathbb{R}^3$, the flow $\Phi^t$ of $v$ preserves convexity of $\Omega$ for small enough times, i.e., $\Phi^{t}(\Omega)$ is convex for all $|t|\ll 1$.

\subsection{The spectral problem on Lipschitz domains}\label{S.Lips}

Since a general convex domain~$\Om$ only has a Lipschitz continuous
boundary, we need to introduce some notation to make sense of the spectral problem~\eqref{P}. The results of this section work for any bounded Lipschitz domain, so they are of independent interest.

The boundary trace $v\restr$ of a function or vector field on $\Om$ is assumed to be taken as the nontangential
limit at almost every boundary point (with respect to the surface measure on $\pd\Om$), whenever this exists. More precisely, it is defined as
\[
v\restr(x):=\lim_{y\to x,\, y\in\Ga(x)} v(y)
\]
for a.e. $x\in\pd\Om$, where $\Ga(x)$ denotes the interior of a regular family of circular
truncated cones $\{\Ga(x): x\in\pd\Om\}$ with vertex at $x$, as defined e.g.\ in~\cite{Verchota}.


The space of harmonic fields on $\Om$ is then defined as
\[
\Harm :=\Big\{ h\in C^\infty(\Om): \Div h=0\,,\; \curl h=0\,,\;
h\restr\cdot N=0\Big\}\,,
\]
which is a vector space whose dimension is equal to the genus of $\pd\Om$, see~\cite[Theorem~11.1]{MMT} for details (in particular for the proof that $h\restr \in L^2(\pd\Om)$, and hence the quantity $h\restr\cdot N$ is well defined). When the domain~$\Om$ is of class~$C^1$, one can understand $h\restr\cdot N$ in the
sense of ordinary traces.

The following theorem shows the self-adjointness of the curl operator for bounded domains. In the case that $\pd\Om$ is $C^2$, see e.g.~\cite{Giga}. The general case of Lipschitz domains follows from~\cite[Section~2]{Picard} (see also~\cite{Filo,Hiptmair}):

\begin{theorem}\label{T.sa}
The curl operator with domain
\[
\Dom:=\big\{ v\in \cK(\Om): \curl v\in \cK(\Om)
\big\}
\]
defines a self-adjoint operator on $\cK(\Om)$ with compact inverse.
\end{theorem}

The following result is also standard, but we provide a proof for the sake of completeness. It establishes the regularity of the eigenfields of curl in $\Dom$, as well as the regularity of their nontangential traces.

\begin{proposition}\label{P.bc}
If $v\in\Dom\cup \Harm$, then $v\in H^{1/2}(\Om)$, $\Div v=0$ in the sense of distributions, $N\cdot
v\restr=0$ and $N\times v\restr\in L^2(\pd\Om)$. Furthermore, if
$v\in\Dom\cup\Harm$ satisfies the equation $\curl v=\mu v$ for some real constant~$\mu$,
then $v\in C^\infty(\Om)$ (in fact, $v$ is analytic).
\end{proposition}

\begin{proof}
Let $v\in\Dom$. Since $v$ is orthogonal in $L^2$ to any curl-free vector field, it
follows that
\[
0=\int_\Om \nabla\vp\cdot v \, dx=- \int_\Om \vp\, \Div v \, dx
\]
for all $\vp\in C^1_c(\Om)$, which implies that $\Div v=0$ (in the sense of distributions). Since
$v$ and $\curl v$ are in $L^2(\Om)$, it follows (see
e.g.~\cite[Section~4]{Buffa}) that the nontangential trace $N\cdot v\restr$ is
well defined in $H^{-1/2}(\pd\Om)$. Taking an arbitrary function
$\psi\in C^1(\BOm)$ and using that $v$ is divergence-free, one finds that
\[
0=\int_\Om v\cdot \nabla\psi\, dx = \int_{\pd\Om}\psi\, N\cdot
v\restr\, dS\,,
\]
which ensures that $N\cdot v\restr=0$. Since for any divergence-free
vector field on~$\Om$ one has~\cite[Theorem 11.2]{MMT}
\[
\|v\|_{H^{\frac12}(\Om)} + \| N\times v\restr\|_{L^2(\pd\Om)} \leq
C\big( \|v\|_{L^2(\Om)}+ \|\curl v\|_{L^2(\Om)} + \|N\cdot v\restr\|_{L^2(\pd\Om)} \big)\,,
\]
the first part of the statement follows. To complete the proof of the
proposition, we observe that a solution of
\[
\curl v= \mu v
\]
also satisfies
\[
\De v+\mu^2 v=0
\]
in~$\Om$, so it is of class $C^\infty(\Om)$ by elliptic
regularity (in fact, $v$ is real-analytic). The same proof works, mutatis mutandis, when $v\in\Harm$.
\end{proof}

\begin{remark}\label{R.regularity}
If the domain~$\Om$ is smooth, it is classical that a vector field
$v\in\Dom$ is of class~$H^1(\Om)$ (and therefore $N\times v\restr\in
H^{1/2}(\pd\Om)$) as one can employ the classical estimate
\[
\|v\|_{H^1(\Om)}\leq C\big( \|v\|_{L^2(\Om)}+ \|\curl v\|_{L^2(\Om)}\big)\,,
\]
using that $\Div v= 0$. If the domain is only Lipschitz,
however the $H^{1/2}$~bound is generally sharp (see e.g.~\cite[p.~87]{MMT}). It should be noticed that, for
all $v\in\Dom$,
\[
|v|^2\restr= |N\times v\restr|^2
\]
is in $L^1(\pd\Om)$.
\end{remark}

Theorem~\ref{T.sa} implies that the spectrum of curl is as described in the introduction. We finish this section noticing that the first positive (or negative) curl eigenvalue $\mu_1(\Om)$ (resp. $\mu_{-1}(\Om)$) is bounded from below by a constant that only depends on the volume $|\Om|$. This is a sort of Faber-Krahn estimate (alveit non-sharp), whose proof is exactly the same (mutatis mutandis) as the proof for $C^2$ domains presented in~\cite{EP21} (one only needs to use the Hodge decomposition for Lipschitz domains, see e.g.~\cite[Proposition~11.3]{MMT}). We state it here for future reference.

\begin{theorem}\label{T.FK}
For any bounded Lipschitz domain $\Om\subset\RR^3$,
\[
\min\{\mu_1(\Om),-\mu_{-1}(\Om)\} \geq \bigg(\frac{4\pi}{3|\Om|}\bigg)^{1/3}\,.
\]
\end{theorem}

\subsection{The spectral problem on convex domains}\label{S.convpre}

In what follows, let $\Om$ be a bounded convex domain. We infer from Theorem~\ref{T.sa} that curl defines a self-adjoint operator with domain $\Dom$ and discrete spectrum. It turns out that the convexity of $\Om$ allows one to improve the regularity obtained for general Lipschitz domains in Proposition~\ref{P.bc}.

\begin{lemma}
	\label{L.rc}
Let $\Om$ be a bounded convex domain. If $v\in\Dom\cup \Harm$ then $v\in H^1(\Om)$. Moreover, if $v$ satisfies the equation $\curl v=\mu v$ in $\Om$ for some real constant $\mu$, then $|v|^2\in W^{1,\frac{3}{2}}(\Omega)$. In particular, $|v|^2\restr\in W^{\frac{1}{3},\frac{3}{2}}(\partial\Omega)$.
\end{lemma}
\begin{proof}
Obviously, $\curl v\in L^2(\Om)$, and $\Div v=0$ and $N\cdot v\restr=0$ by Proposition~\ref{P.bc}. The convexity of $\Omega$ then implies that $v\in H^1(\Omega)$, cf.~\cite[Theorem 2.17]{ABDG98}. Now assume that $\curl v=\mu v$ for some $\mu\in\RR$. It follows from the well known identity
\[
\frac12 \nabla|v|^2=\nabla_vv+v\times\curl v=\nabla_vv
\]
and H\"older inequality, that
\[
\frac12\|\nabla|v|^2\|_{L^\frac{3}{2}(\Om)}= \|\nabla_vv\|_{L^\frac{3}{2}(\Om)}\leq C\|v\|_{L^6(\Om)}\|\nabla v\|_{L^2(\Om)}\,.
\]
Noticing that $v\in L^6(\Om)$ by Sobolev embedding, so in particular $|v|^2\in L^3(\Omega)$, we infer that $|v|^2\in W^{1,\frac{3}{2}}(\Omega)$. The standard trace inequality implies the claim about the trace regularity.
\end{proof}

The last result of this section shows the equivalence between strongly convex and stably convex domains, provided that they are regular enough. This characterization is relevant in view of Proposition~\ref{MT3} in Section~\ref{analytic}.

\begin{proposition}\label{L.stable}
Let $\Om$ be a bounded domain of class $C^2$. Then it is strongly convex if and only if it is stable. Moreover, a $C^{1,1}$ strongly convex domain is stable.
\end{proposition}
\begin{proof}
The first claim is obvious because for $C^2$ domains, strong and stable convexity are both characterized by positivity of the Gauss curvature. To prove the second claim, let $\Om$ be a bounded strongly convex domain and $f\in C^{1,1}(\RR^3)$ a defining function (there is no loss of generality in assuming that the Lipschitz constant of $\nabla f$ is uniform on $\RR^3$). If $\Phi^t$, $|t|\leq \delta$, is the flow of a smooth compactly supported vector field on $\RR^3$, it is obvious that the domain $\Phi^{-t}(\Om)$ is defined by the function $f\circ \Phi^t\in C^{1,1}(\RR^3)$. We claim that $f\circ \Phi^t$ is strongly convex if $\delta$ is small enough. Indeed, since $f$ is strongly convex we can write
\[
(\nabla f(\Phi^t(x))-\nabla f(\Phi^t(y)))\cdot (\Phi^t(x)-\Phi^t(y))\geq \sigma |\Phi^t(x)-\Phi^t(y)|^2
\]
for some $\sigma>0$ and all $x,y\in\RR^3$. Using that
\[
|\Phi^t(x)-\Phi^t(y)|\geq |x-y|-|(\Phi^t(x)-x)-(\Phi^t(y)-y)|
\]
and the estimate
\[
|(\Phi^t(x)-x)-(\Phi^t(y)-y)|\leq \text{sup}_{\RR^3}\|D\Phi^t-I\|\cdot|x-y|\leq C\delta |x-y|\,,
\]
which follows from the mean value theorem, we infer that
\begin{equation}\label{e1}
(\nabla f(\Phi^t(x))-\nabla f(\Phi^t(y)))\cdot (\Phi^t(x)-\Phi^t(y))\geq \sigma (1-C\delta)^2 |x-y|^2\geq \frac{\sigma}{2}|x-y|^2
\end{equation}
if $\delta$ is small enough.
Next, noticing that
\[
\nabla f(\Phi^t(x))=\big[(D\Phi^t(x))^T\big]^{-1}\nabla (f\circ \Phi^t)(x)
\]
and
\[
\big[(D\Phi^t)^T\big]^{-1}=I+O(\delta)\,,
\]
we conclude that
\begin{align}
\nonumber (\nabla f(\Phi^t(x))-\nabla f(\Phi^t(y)))&\cdot (\Phi^t(x)-\Phi^t(y))\leq (\nabla (f\circ \Phi^t)(x)-\nabla (f\circ \Phi^t)(y))\cdot (x-y)
\\ \nonumber&+C\delta|\nabla (f\circ \Phi^t)(x)-\nabla (f\circ \Phi^t)(y)|\cdot|x-y|\leq \\
&(\nabla (f\circ \Phi^t)(x)-\nabla (f\circ \Phi^t)(y))\cdot (x-y)+CL\delta|x-y|^2 \,,\label{e2}
\end{align}
where we have used again that $\delta$ is small, and the last inequality follows from the Lipschitz estimate
\[
|\nabla (f\circ \Phi^t)(x)-\nabla (f\circ \Phi^t)(y)|\leq L|x-y|
\]
for all $x,y\in \RR^3$ and some $L>0$.
Putting together Equations~\eqref{e1} and~\eqref{e2} we infer that
\[
(\nabla (f\circ \Phi^t)(x)-\nabla (f\circ \Phi^t)(y))\cdot (x-y)\geq \frac{\sigma}{4}|x-y|^2
\]
for all $x,y\in \RR^3$, which completes the proof of the proposition.
\end{proof}

\section{Existence of optimal convex domains}\label{proof}

In this section we establish the first part of Theorem~\ref{T1}. Our proof follows the reasoning presented in~\cite[Theorem 2.4.1]{Hen06} for the first eigenvalue of the Dirichlet Laplacian. To achieve this we first show the continuity of the first curl eigenvalue with respect to the Hausdorff metric, and then we prove a new estimate for the curl operator implying that the first curl eigenvalue is not uniformly bounded for sequences of convex domains whose diameter tends to infinity.

We recall that if $K_1,K_2\subset \mathbb{R}^3$ are (non-empty) compact sets then the Hausdorff distance $d_H$ between them is defined as
$$d_H(K_1,K_2):=\max\{\sup_{x\in K_1}d(x,K_2),\sup_{y\in K_2}d(y,K_1) \}\,,$$
where $d(x,K_1):=\min_{y\in K_1}|x-y|$.

\subsection{Step~1: Continuity and monotonicity of the eigenvalues}

We first prove two elementary properties of the first curl eigenvalue, which are reminiscent of the properties of the Dirichlet eigenvalues of the Laplacian. In the statement we use the notation $\lambda \Om:=\{\lambda x:\,x\in \Om\}$ for the scaling of a domain $\Om$ for some $\lambda>0$.
\begin{lemma}[Scaling property]
	\label{L.Scaling}
	Let $\Omega\subset \mathbb{R}^3$ be a bounded Lipschitz domain and $\lambda$ a positive constant. Then $\mu_1(\lambda \Omega)=\frac{\mu_1(\Omega)}{\lambda}$.
\end{lemma}
\begin{proof}
Let $u_1\in \mathcal{D}_{\Om}$ be any eigenfield corresponding to the first (positive) curl eigenvalue $\mu_1$. It is then elementary to check that $\hat{u}_1(x):=u_1\left(\frac{x}{\lambda}\right)$ is in the functional space $\mathcal{D}_{\lambda \Omega}$ and $\operatorname{curl}(\hat{u}_1)=\frac{\mu_1}{\lambda}\hat{u}_1$. We then conclude that the first (positive) eigenvalue of $\operatorname{curl}$ on $\lambda \Omega$ is $\frac{\mu_1}{\lambda}$ as claimed.
	\end{proof}
\begin{lemma}[Monotonicity principle]
	\label{L.MP}
	Let $\Omega_1\subseteq \Omega_2\subset \mathbb{R}^3$ be bounded Lipschitz domains. Then the first (positive) curl eigenvalues $\mu_1(\Omega_i)$, $i=1,2$, satisfy
	\[
	\mu_1(\Omega_1)\geq \mu_1(\Omega_2)\,,
	\]
and equality holds if and only if $\Omega_1=\Omega_2$.
\end{lemma}
\begin{proof}
Take any vector field $v\in \mathcal{K}(\Omega_1)$. We can then define the vector field
\[
\hat{v}:=\begin{cases}
	v&\text{on }\Omega_1\,,\\
	0&\text{on }\Omega_2\setminus \Omega_1\,,
\end{cases}
\]
which is obviously in $\mathcal{K}(\Omega_2)$. The first part of the lemma then follows from the variational characterization of the first curl eigenvalue (see e.g.~\cite{Gerner}):
\begin{equation*}
\mu_1(\Omega_2)=\inf_{w\in \mathcal{K}(\Omega_2)\text{, }\mathcal{H}(w)>0}\frac{\|w\|^2_{L^2(\Omega_2)}}{\mathcal{H}(w)}\leq \inf_{v\in \mathcal{K}(\Om_1)\text{, }\mathcal{H}(v)>0}\frac{\|\hat{v}\|^2_{L^2(\Omega_2)}}{\mathcal{H}(\hat{v})}
\end{equation*}
\begin{equation}
\label{E.Var}
=\inf_{v\in \mathcal{K}(\Om_1)\text{, }\mathcal{H}(v)>0}\frac{\|v\|^2_{L^2(\Omega_1)}}{\mathcal{H}(v)}=\mu_1(\Omega_1)
\end{equation}
by definition of $\hat{v}$. Here $\mathcal{H}(w)$ is the helicity of $w$, which is defined as
\[
\cH(w):=\int_{\Om_2} \curl^{-1}w\cdot w dx\,,
\]
where $\curl^{-1}$ is the compact self-adjoint operator on $\cK(\Om_2)$ defined by the inverse of the curl. Since $v\in\cK(\Om_1)$, it is straightforward to check that $\cH(\hat v)=\cH(v)$, which is used in Equation~\eqref{E.Var}.

Now suppose that $\mu_1(\Omega_1)=\mu_1(\Omega_2)$ and $\overline\Om_1\subset\Om_2$. Let $u_1\in \mathcal{D}_{\Om_1}$ be a curl eigenfield corresponding to $\mu_1(\Omega_1)$. Then $\hat{u}_1\in \mathcal{K}(\Omega_2)$ defined as above realizes the minimum in the variational characterization of $\mu_1(\Om_2)$ and hence it is an eigenfield of curl in $\mathcal{D}_{\Om_2}$ corresponding to $\mu_1(\Om_2)$. Since $\hat{u}_1$ is real analytic in $\Omega_2$ by Proposition~\ref{P.bc}, and $\hat{u}_1=0$ on $\Omega_2\setminus \Omega_1$, we conclude that $\hat u_1$ is zero everywhere. Accordingly $\Om_1=\Om_2$.
\end{proof}

The main result of this section is the following proposition, which is standard in the case of the Dirichlet Laplacian. For this we denote by $\mathcal{C}(\mathbb{R}^3)$ the set of all bounded convex domains of $\mathbb{R}^3$
\begin{equation}
	\label{ConDef}
\mathcal{C}(\mathbb{R}^3):=\{\Om\subset \mathbb{R}^3:\Om\text{ bounded convex domain} \}\,,
\end{equation}
and the Hausdorff distance is given by $d_H(\Om_1,\Om_2)=d_H(\pd\Om_1,\pd\Om_2)$, cf.~\cite{Wills}. It is standard that $(\mathcal{C}(\mathbb{R}^3),d_H)$ is then a metric space. Further, it is easy to check that if $\{\Omega_n\}\subset \mathcal{C}(\mathbb{R}^3)$ is a sequence satisfying $|\Omega_n|=V$ for some constant $V>0$ and all $n$, and $\Omega_n$ converges to a bounded convex domain $\Omega$ with respect to $d_H$, then $|\Omega|=V$.

\begin{proposition}\label{C.c}
The assignment of the first (positive) curl eigenvalue
$$\mu_1:(\mathcal{C}(\mathbb{R}^3),d_H)\rightarrow (0,\infty)$$
is continuous.
\end{proposition}
\begin{proof} Let us assume that the sequence $\{\Omega_n\}\subset \mathcal{C}(\mathbb{R}^3)$ converges to some $\Omega\in \mathcal{C}(\mathbb{R}^3)$ in the Hausdorff metric. It then follows, cf. \cite[Lemma 3.3]{CoFi10}, that we can find two sequences $\{t_n\},\{s_n\}\subset \mathbb{R}$ both converging to $1$ such that $t_n\Omega_n\subseteq \Omega\subseteq s_n\Omega_n$. The monotonicity and scaling properties of $\mu_1$ (Lemmas~\ref{L.Scaling} and~\ref{L.MP}) then imply
\[
\frac{\mu_1(\Omega_n)}{s_n}=\mu_1(s_n\Omega_n)\leq \mu_1(\Omega)\leq \mu_1(t_n\Omega_n)=\frac{\mu_1(\Omega_n)}{t_n}.
\]
As $t_n$ and $s_n$ converge to $1$, we may take the $\limsup$ on the left side and the $\liminf$ on the right side to conclude $\mu_1(\Omega)=\lim_{n\rightarrow+\infty}\mu_1(\Omega_n)$, which proves the proposition.
\newline
\end{proof}

\subsection{Step~2: An estimate for sequences of convex domains}
As usual, the diameter of a subset $\Omega\subset \mathbb{R}^3$ is
$$\operatorname{diam}(\Omega):=\sup_{x,y\in \Omega}|x-y|\,.$$
In this step we show that if the diameter of a sequence of convex domains of the same volume tends to infinity then the first curl eigenvalue tends to infinity as well. A key ingredient to prove this result is Lemma~\ref{L.trap} in Section~\ref{app.trapping}, which proves that the aforementioned sequence of domains can be enclosed into a sequence of cylinders whose heights tend to zero and radii tend to infinity in a controlled way. The monotonicity principle then allows us to reduce the problem to study the first curl eigenvalue along the aforementioned sequence of cylinders. While this is a simple task in the case of the Dirichlet Laplacian, because the spectrum can be explicitly computed~\cite[Chapter 1.2.5]{Hen06}, in the case of the curl operator, computing the first curl eigenvalue is notoriously difficult. The main issue is that one cannot easily derive decoupled elliptic equations for the component functions.

\begin{lemma}\label{L5}
Let $\{\Omega_n\}\subset \mathbb{R}^3$ be a sequence of bounded convex domains with $|\Omega_n|=V$ for some $V>0$ and such that $\operatorname{diam}(\Omega_n)\rightarrow \infty$ as $n\rightarrow \infty$. Then $\mu_1(\Omega_n)\rightarrow \infty$.
\end{lemma}
\begin{proof} 
Lemma~\ref{L.trap} below shows that, up to isometry, we can trap the domains $\Om_n$ into a sequence of cylinders
$$C_n:=D_{R_n}\times (-h_n,h_n)$$
of height $2h_n\to 0$ and of radius $R_n\to\infty$ which satisfy the relation $R_nh^2_n\leq c$ for some uniform constant $c>0$.
Then Lemma~\ref{L.MP} implies that
\[
\mu_1(\Omega_n)\geq \mu_1(C_n)\,.
\]
We claim that the first curl eigenvalue of the cylinder $C_n$ diverges as $n\to\infty$.
Indeed, if $\operatorname{curl}^{-1}$ denotes the compact inverse of curl, then $$\int_{C_n}\operatorname{curl}^{-1} u_1\cdot u_1 dx=\frac{1}{\mu_1(C_n)}\|u_1\|^2_{L^2}$$
for any eigenfield $u_1\in \mathcal{D}_{C_n}$ associated to $\mu_1(C_n)$. Defining the Biot-Savart operator~\cite{JMPA}
\[
\operatorname{BS}(u_1)(x):=\frac{1}{4\pi}\int_{C_n}\frac{u_1(y)\times (x-y)}{|x-y|^3}dy\,,
\]
and using that $u_1$ is $L^2$-orthogonal to the curl-free fields, we find
\begin{equation}
	\label{FEE}
\frac{\|u_1\|^2_{L^2}}{\mu_1(C_n)}=\int_{C_n} \operatorname{BS}(u_1)\cdot u_1 dx\leq \|u_1\|_{L^2}\|\operatorname{BS}(u_1)\|_{L^2}\,.
\end{equation}

To bound the $L^2$ norm of the Biot-Savart operator, we use Cauchy-Schwarz to write
\[
|\operatorname{BS}(u_1)(x)|\leq \frac{1}{4\pi}\Big(\int_{C_n}\frac{|u_1(y)|^2}{|x-y|^2}dy\Big)^{1/2}\Big(\int_{C_n}\frac{1}{|x-y|^2}dy\Big)^{1/2}\,,
\]
and then
\[
\|\operatorname{BS}(u_1)\|^2_{L^2}\leq \frac{M^2}{16\pi^2}\|u_1\|^2_{L^2}\,.
\]
Here the constant $M$ is defined as $M:=\sup_{x\in C_n}\int_{C_n}\frac{1}{|x-y|^2}dy$.

This estimate combined with Equation~\eqref{FEE} imply that
$$\mu_1(C_n)\geq \frac{4\pi}{M}\,.$$
It can be checked~\cite{OHara} (due to the symmetry of the cylinders) that $M=\int_{C_n}\frac{1}{|y|^2}dy$, and a simple integration yields
\[
M=2\pi h_n \ln\left(1+\frac{R^2_n}{h^2_n}\right)+4\pi R_n\arctan\left(\frac{h_n}{R_n}\right)\,.
\]
Finally, using the relation $R_nh^2_n\leq c$ and that $R_n\rightarrow \infty$, we conclude that $M\to 0$ as $n\rightarrow \infty$ and consequently $\mu_1(C_n)\to \infty$. This completes the proof of the lemma.
\end{proof}

\subsection{Step~3: A geometric lemma}\label{app.trapping}

Now we prove that a sequence of convex domains of fixed volume and diameter tending to infinity can be trapped inside a sequence of cylinders of heights that tend to zero. The proof adapts to $\RR^3$ the 2-dimensional argument presented in~~\cite[Theorem 2.4.1]{Hen06}.

\begin{lemma}\label{L.trap}
Let $\{\Omega_n\}\subset \mathbb{R}^3$ be a collection of bounded convex domains of the same volume $V>0$. If $d_n:=\operatorname{diam}(\Omega_n)\rightarrow \infty$ as $n\rightarrow \infty$, then there exists a sequence $h_n>0$ such that $h^2_nd_n\leq c$ for some uniform constant $c>0$ and $\Omega_n\subset D_{2d_n}\times(-h_n,h_n)$ (up to isometry).
\end{lemma}
\begin{proof}
Let $L_3\subset \Om_n$ be a line segment that realizes the diameter of $\Om_n$, i.e., $|L_3|=d_n$. We can then consider planes perpendicular to $L_3$ which intersect $\Om_n$. There will be (at least) a plane $H$ such that $\operatorname{diam}(\Om_n\cap H)$ becomes maximal among all such planes. Take again a line segment $L_2\subset\Om_n\cap H$ which realizes $\operatorname{diam}(\Om_n\cap H)$. Since $L_2$ and $L_3$ span a plane, we can consider line segments perpendicular to this plane and intersecting $\Om_n$. As before, there is a line segment $L_1$ which has maximal length among all such lines. Obviously $|L_1|\leq |L_2|$. For notational simplicity, we are omitting the $n$-dependence of $L_1,L_2,L_3$. We claim that there is a uniform constant $c>0$ such that
$$|L_1|^2|L_3|\leq c\,.$$

We argue by contradiction and assume that $|L_1|^2|L_3|$ diverges as $n\to \infty$ (and hence $|L_2|^2|L_3|$ also diverges). Consider line segments contained in the plane $H$ and orthogonal to $L_2$, and denote by $L^\perp_2\subset \Om_n\cap H$ any of these segments with maximal length. Similarly we can consider the plane $H^\prime$ containing $L_1$ and perpendicular to $L_3$ and denote by $L^\perp_1\subset \Om_n\cap H^\prime$ a maximal length line segment. If we connect the end points of $L_2$ and $L^\perp_2$ with the end points of $L_3$, we obtain a pyramid (see Figure~\ref{FigurePyramid}) which, by convexity, is inscribed in $\Om_n$, and hence
$$|L_2||L^\perp_2||L_3|\leq 6V\,.$$
Since $|L_2|\geq |L^\perp_2|$ we infer that $|L^\perp_2|^2|L_3|\leq 6V$. Analogously we obtain
$$|L_1||L^\perp_1||L_3|\leq 6V\,.$$

\begin{figure}[htbp]
		\centering
		\includegraphics[width=1.0\textwidth]{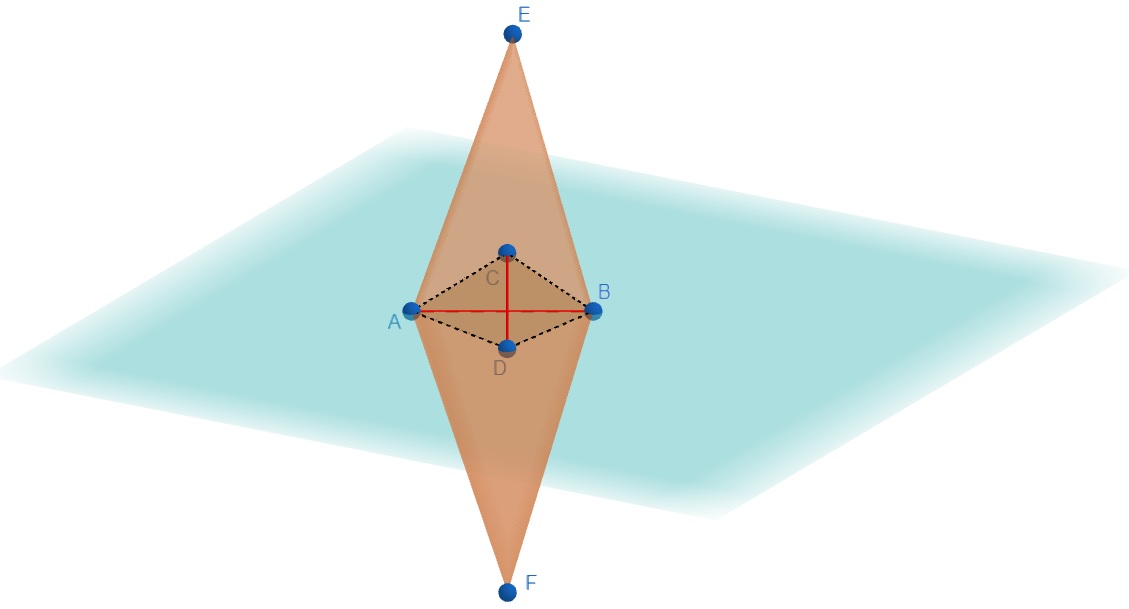}
		\caption{Connecting the end points of the line segments $L_2=\overline{AB}$ and $L^\perp_2=\overline{CD}$ with the end points of the diameter $L_3=\overline{EF}$.}
		\label{FigurePyramid}
\end{figure}

If $|L^\perp_1|\geq |L_1|$ we deduce that $|L_1|^2|L_3|\leq 6V$, which contradicts our blow up assumption. Hence we must have $|L^\perp_1|\leq |L_1|$ and we deduce that
$$|L^\perp_1|^2|L_3|\leq 6V\,.$$
By the definition of $L_1$, if it were contained in the plane spanned by $L_2$ and $L^\perp_2$, then we could have chosen $L^\perp_2=L_1$, which would again yield a contradiction with our blow up assumption. Consequently $L_1$ and $L_2$ must be contained in two distinct parallel planes perpendicular to $L_3$.

The aforementioned planes divide the diameter in $3$ parts (which might be degenerate if one of the endpoints of $L_3$ is contained in one of the planes). Denote by $b_n$ the distance between these parallel planes and $x_n,y_n\in L_3$ the end points of $L_3$. Suppose that starting at $x_n$ and running along $L_3$ we first intersect the plane containing $L_2$ (otherwise, one can start at $y_n$), and denote the distance from $x_n$ to this plane by $a_n$ and the distance from $y_n$ to the other plane by $c_n$. Connecting $x_n$ with the end points of $L_1$, a simple application of the basic proportionality theorem to the resulting triangle, allows us to obtain
\[
\frac{a_n}{b_n}\leq \frac{|L^\perp_2|}{|L_1|-|L^\perp_2|}=\frac{|L^\perp_2||L_3|^{1/2}}{|L_1||L_3|^{1/2}-|L^\perp_2||L_3|^{1/2}}\leq\frac{(6V)^{1/2}}{|L_1||L_3|^{1/2}-(6V)^{1/2}}\,,
\]
which implies, by our blow up assumption, that the ratio $\frac{a_n}{b_n}$ tends to zero. A similar argument shows that the ratio $\frac{c_n}{b_n}$ also tends to zero. Since $d_n=a_n+b_n+c_n$ we conclude that $\frac{b_n}{d_n}$ converges to $1$ as $n\rightarrow\infty$.

Finally, let us consider the polytope $\mathcal{P}$ obtained by connecting the end points of $L_1$ with the end points of $L_2$, which is inscribed in $\Om_n$ by convexity, and whose volume is given by $|\mathcal{P}|=\frac{b_nL_1L_2}{6}$. As the ratio of $b_n$ and $d_n$ tends to $1$ and $|L_3|=d_n$, we find for large enough $n$,
$$V\geq \frac{|L_3||L_1||L_2|}{12}\geq \frac{|L_1|^2|L_3|}{12}\to \infty\,,$$
which is a contradiction, thus implying that $|L_1|^2|L_3|\leq c$ for some uniform $c>0$, as we wanted to show.

Applying an Euclidean motion if necessary, it is then obvious that
$$\Om_n\subset (-d_n,d_n)^2\times (-|L_1|,|L_1|) \subset D_{2d_n}\times(-|L_1|,|L_1|)\,,$$
thus completing the proof of the lemma.
\end{proof}

\subsection{Step~4: Completing the proof}

Let $\{\Omega_n\}$ be a minimizing sequence with fixed volume $V>0$. According to Lemma~\ref{L5} we may assume that the sequence of diameters must be uniformly bounded, and hence there is $R>0$ so that $\{\Omega_n\}\subset B_R$ for all $n$. A simple application of Blaschke's selection theorem implies that there is a subsequence converging to some bounded convex domain $\Omega\subset B_R$ in the Hausdorff sense with $|\Omega|=V$. The theorem follows using the continuity of the first curl eigenvalue, cf. Proposition~\ref{C.c}:
$$\mu_1(\Omega)=\lim_{n\rightarrow\infty}\mu_1(\Omega_n)\,.$$

\subsection{Axially symmetric convex domains}\label{SS.axi}
In this final subsection we consider bounded convex domains $\Om$ that are axially symmetric with respect to some line (up to isometry, there is no loss of generality in assuming that the symmetry line is the $z$ axis). For each positive constant $V$ we show that there exist optimal domains of volume $V$ for the first curl eigenvalue in the aforementioned class of convex sets. The proof is essentially the same as in the previous subsections.

\begin{theorem}\label{T.RotE}
Let $V$ be a positive real number. Then there exists a bounded axially symmetric convex domain $\Om$ of volume $V$ which minimizes $\mu_1$ among all other domains in the same class with the same volume.
\end{theorem}
\begin{remark}
This optimal domain does not need to minimize $\mu_1$ among all other convex domains of the same volume. Conversely the optimal domain in the first part of Theorem~\ref{T1} does not need to be axially symmetric. Note also that in the axially symmetric setting there is no analogue of the regularity result proved in Proposition~\ref{P.regularity} below because the class of flows $\Phi^t$ that preserve axial symmetry is too small.
\end{remark}

\begin{proof}
Let $\Om_n$ be a minimizing sequence in the class of bounded axially symmetric convex domains of volume $V$. Lemma~\ref{L5} implies that we can take all $\Om_n$ to be contained in some ball $B_R$ for large enough radius $R>0$. It follows from Blaschke's selection theorem that, up to a subsequence, the domains $\Om_n$ converge to a bounded convex domain $\Om$ of volume $V$ in the Hausdorff distance $d_H$. By Proposition~\ref{C.c} we know that $$\mu_1(\Om)=\lim_{n\rightarrow\infty}\mu_1(\Om_n)\,.$$
Now, for each $\phi\in[0,2\pi)$, let $R_\phi$ be the rotation of $\RR^3$ defined as
\[
(x,y,z)\mapsto (\cos\phi\, x + \sin\phi\, y,-\sin\phi\, x+\cos\phi\, y,z)\,.
\]
By assumption, we know that $R_{\phi}(\Om_n)=\Om_n$ for all $\phi$ and $n$. Accordingly, since the Hausdorff distance is preserved by isometries, we can write
\[
d_H(\Om_n,R_\phi(\Om))=d_H(R_\phi(\Om_n),R_\phi(\Om))=d_H(\Om_n,\Om)\to 0
\]
as $n\to\infty$, so we conclude that $\Om_n$ converges to $R_\phi(\Om)$ for all $\phi\in[0,2\pi)$. The uniqueness of the Hausdorff limit implies that $R_\phi(\Om)=\Om$ and hence $\Om$ is axially symmetric, as we wanted to show.
	\end{proof}

\section{Non-analyticity of optimal convex domains}\label{analytic}

In this section we prove the second part of Theorem~\ref{T1}, i.e., we show that the optimal convex domains of volume $V>0$ are not analytic. Key to prove this result is a lemma showing that the pointwise norm of any eigenfunction of curl associated to the first eigenvalue $\mu_1(\Om)$ is constant on the stably convex part of the boundary $\pd\Om$, see Lemma~\ref{L.norm} below.

\begin{proposition}\label{P.regularity}
Let $\Om$ be an optimal convex domain of volume $V$. Then its boundary $\pd\Om$ is not analytic.
\end{proposition}
\begin{proof}
Let us assume that $\pd\Om$ is analytic. It is easy to check that there is an open and dense set $U\subset\pd\Om$ where $\pd\Om$ is stably convex (which corresponds to the piece of $\pd\Om$ whose Gauss curvature is positive). Now let $v$ be any compactly supported $C^\infty$ vector field on $\mathbb{R}^3$ which is zero on a neighborhood of $\partial\Om\setminus U$. The stable convexity then implies that, for small enough times $t$, its flow $\Phi^t$ preserves the convexity of $\Om$.

If $u_1$ is a curl eigenfield on $\Om$ associated to the first curl eigenvalue, it follows from Lemma~\ref{L.norm} below that
\begin{equation}\label{eq.a}
|u_1|^2=\frac{\|u_1\|^2_{L^2(\Om)}}{3|\Om|}
\end{equation}
on an open set of $\pd\Om$ strictly contained in $U$. In fact, the analyticity of the boundary implies that $u_1|_{\pd\Om}$ is real analytic~\cite[Theorem A.1]{ELP21}, and therefore we conclude that $|u_1|^2$ is constant everywhere on $\partial\Om$. Since $\pd\Om$ is diffeomorphic to $\mathbb S^2$ by convexity, the Poincar\'{e}-Hopf theorem implies that $u_1$ must vanish somewhere on $\partial\Om$ and hence $u_1=0$ on $\pd\Om$, and in turn $u_1\equiv 0$ on $\Om$ by Equation~\eqref{eq.a}. This contradiction completes the proof of the proposition.
\end{proof}

Before showing that the norm of the first eigenfield $|u_1|$ is constant on the stably convex part of the boundary of an optimal domain, we first prove an auxiliary result on the equivalence between our optimal domain problem and a minimization problem without a volume constraint:
\begin{lemma}\label{L.Equiv}
A bounded convex domain $\Om$ is optimal for the first curl eigenvalue if and only if $\Om$ minimizes the product $|\Om|^{1/3}\mu_1(\Om)$ among all bounded convex domains (with not necessarily the same volume).
\end{lemma}
\begin{proof}
If $\Om$ minimizes $|\Om|^{1/3}\mu_1(\Om)$ among all bounded convex domains, then for any bounded convex domain $\Om^\prime$ with $|\Om^\prime|=|\Om|$ we find
	\[
	|\Om|^{1/3}\mu_1(\Om)\leq |\Om^\prime|^{1/3}\mu_1(\Om^\prime)=|\Om|^{1/3}\mu_1(\Om^\prime)
	\]
and hence $\Om$ is an optimal domain for the first curl eigenvalue. On the other hand, if $\Om$ is a minimizer of the first curl eigenvalue among all bounded convex domains with the same volume, and $\Om^\prime$ is any other bounded convex domain (of possibly different volume), we notice that the scaling $\la \Om^\prime$ with $\lambda:=\Big(\frac{|\Om|}{|\Om^\prime|}\Big)^{1/3}$ satisfies $|\lambda\Om^\prime|=|\Om|$. Consequently, $\mu_1(\Om)\leq \mu_1(\lambda \Om^\prime)$ and the lemma follows from Lemma~\ref{L.Scaling}.
	\end{proof}

We are now ready to prove the key lemma that is used in the proof of Proposition~\ref{P.regularity}. In the statement we use the notation $\Gamma_S$ for the stable part of the boundary of $\Om$, which is an open set. By stable set we mean that the domain $\Phi^t(\Om)$ is convex for all $|t|$ small enough if $\Phi^t$ is the flow defined by a vector field whose support on $\pd\Om$ is contained in $\Gamma_S$. We recall that the trace $|u|^2\restr$ is in the space $W^{\frac13,\frac32}(\pd\Om)$ by Lemma~\ref{L.rc}.

\begin{lemma}\label{L.norm}
Let $\Omega$ be an optimal convex domain and let $v$ be a smooth compactly supported vector field on $\mathbb{R}^3$ whose flow $\Phi^t$ preserves the convexity of $\Omega$ for small enough times. Then, for any first eigenfield $u_1$, we have
	\begin{gather}
		\label{E.norm}
	\int_{\partial\Omega}\left(|u_1|^2-\frac{\|u_1\|^2_{L^2(\Omega)}}{3|\Omega|}\right)v\cdot N\ dS=0\,.
	\end{gather}
In particular $|u_1|^2=\frac{\|u_1\|^2_{L^2}}{3|\Omega|}$ a.e. on $\Gamma_S$, and if $\Omega$ is stably convex then $|u_1|^2=\frac{\|u_1\|^2_{L^2}}{3|\Omega|}$ a.e. on the whole $\partial\Omega$.
\end{lemma}
\begin{proof}
For small $|t|$ we define the domain $\Om^t:=\Phi^t(\Om)$ and the vector field
\[
u^t_1:=\frac{(\Phi^t)_{*}u_1}{\det\left(D\Phi^t\right)\circ\Phi^{-t}}\,,
\]
on $\Om^t$, which is not generally a first eigenfield for $t\neq 0$; obviously $u^t_1\in H^1(\Om^t)$ because $u_1\in H^1(\Om)$ by Lemma~\ref{L.rc}. We claim that $u^t_1\in \mathcal{K}(\Omega^t)$. Indeed, we know from Proposition~\ref{P.bc} that $N\cdot u_1\restr=0$, and since $u^t_1$ is the push-forward of $u_1$ up to a proportionality factor, we easily infer that $N_t\cdot u^t_1\restrt=0$, where $N_t$ is a unit normal vector at almost every point of $\pd\Om^t$ ($u^t_1$ understood as an element in $H^{1/2}(\pd\Om^t)$). The fact that $\Div u^t_1=0$ on $\Om^t$ follows from the following computation using the Lie derivative, with $\mu$ the Euclidean volume form:
\begin{align*}
(\Phi^t)^* (L_{u_1^t}\mu) = L_{(\Phi^t)^*u_1^t}(\Phi^t)^*\mu=L_{u_1}\mu=0
\end{align*}
where we have used that $(\Phi^t)^*\mu = \det\left(D\Phi^t\right) \mu $ and $(\Phi^t)^*u_1^t=\frac{u_1}{\det\left(D\Phi^t\right)}$. Since $\Om^t$ is homeomorphic to a ball, the space of harmonic fields is trivial, i.e., $\cH_{\Om^t}=0$, and hence the Hodge decomposition theorem for Lipschitz domains, cf.~\cite[Proposition~11.3]{MMT}, implies that any curl-free vector field $w\in L^2(\Om^t)$ is of the form $w=\nabla \psi$ for some $\psi\in H^1(\Om^t)$. Accordingly,
\[
\int_{\Om^t}u^t_1\cdot \nabla \psi dx=\int_{\pd\Om^t}\psi N_t \cdot u_1^t\restrt dS=0\,,
\]
thus completing the proof of the claim.

Now we show that the helicity $\mathcal H_{\Om^t}(u_1^t)=\int_{\Om^t} \curl^{-1} u_1^t\cdot u_1^t dx$, where $\curl^{-1}$ is the compact inverse of $\curl$ on $\Om^t$, does not depend on $t$, i.e.,
\[
\mathcal H_{\Om^t}(u_1^t)=\cH(u_1)\,.
\]
Indeed, if $\alpha^t$ is the metric dual $1$-form of $\curl^{-1} u_1^t$, we know by definition of $\curl$ in the language of forms that $i_{u_1^t}\mu=d\alpha^t$;
a straightforward computation using the definition of $u_1^t$ then implies that $\alpha^t=(\Phi^t)_*\alpha_1+d\psi_t$ for some function $\psi_t\in H^1(\Om^t)$. Accordingly
\begin{align*}
\mathcal H_{\Om^t}(u_1^t)&=\int_{\Om^t}\alpha^t\wedge d\alpha^t=\int_{\Om}\alpha_1\wedge d\alpha_1+\int_{\Om^t}d\psi_t\wedge d\alpha^t\\
&=\cH(u_1)+\int_{\Om^t} u_1^t\cdot \nabla \psi_t dx=\cH(u_1)\,.
\end{align*}

By Lemma~\ref{L.Equiv} and the assumption that $\Om^t$ is convex for all small enough $|t|$, the function $|\Omega^t|^{1/3}\mu_1(\Omega^t)$ must have a local minimum at $t=0$. We can now define the smooth function
\[
f(t):=|\Omega^t|^{1/3}\frac{\|u^t_1\|^2_{L^2(\Omega^t)}}{\mathcal{H}(u_1)}=|\Omega^t|^{1/3}\frac{\|u^t_1\|^2_{L^2(\Omega^t)}}{\mathcal{H}_{\Omega^t}(u^t_1)} \geq |\Omega^t|^{1/3} \mu_1(\Om^t)\,,
\]
where we have used the variational principle~\eqref{E.Var} in the last inequality. Therefore
\[
f(0)=|\Om|^{1/3}\mu_1(\Om)\leq |\Om^t|^{1/3}\mu_1(\Om^t)\leq f(t)
\]
for all $|t|$ small enough, and hence $f$ attains a local minimum at $t=0$ as well, in particular, $f'(0)=0$.

Noticing that
$$\frac{d |\Om^t|}{dt}\Big|_{t=0}=\int_\Om\Div v \,dx=\int_{\partial\Om}N\cdot v\,dS\,,$$
we can write
\[
f'(0)=\frac{|\Om|^{-2/3}\|u_1\|^2_{L^2(\Om)}}{3\cH(u_1)} \int_{\pd\Om}N\cdot v\, dS+\frac{|\Om|^{1/3}}{\cH(u_1)}\,\frac{d}{dt}\|u_1^t\|^2_{L^2(\Om^t)}\Big|_{t=0}\,.
\]
On the other hand, it is easy to check from the definition of $u_1^t$ that
\[
\frac{d u_1^t}{dt}\Big|_{t=0}=[u_1,v]-(\Div v) u_1=\curl(v\times u_1)\,,
\]
where $[\cdot,\cdot]$ denotes the Lie bracket of vector fields and to pass to the last equality we have used the well known identity for $\curl(v\times u_1)$ in terms of the Lie bracket. This allows us to compute the $t$ derivative of $\|u_1^t\|^2_{L^2(\Om^t)}$:
\begin{align*}
\frac{d}{dt}\|u_1^t\|^2_{L^2(\Om^t)}\Big|_{t=0}&=2\int_{\Om}u_1\cdot\curl(v\times u_1)dx+\int_{\pd\Om}|u_1|^2N\cdot v dS\\
&=-\int_{\pd\Om}|u_1|^2N\cdot v dS\,.
\end{align*}
To obtain the last equality we have integrated by parts and used that $\curl u_1=\mu_1 u_1$. We recall that the boundary terms are well defined because $|u_1|^2\restr\in W^{\frac13,\frac32}(\pd\Om)$ on account of Lemma~\ref{L.rc}.

Equation~\eqref{E.norm} then follows after plugging this expression into the formula for $f'(0)=0$. If $\Gamma_S$ is the stable part of $\pd\Om$ it is clear that the convexity of $\Om$ is preserved by the local flow of any vector field whose support on $\pd\Om$ is contained in $\Gamma_S$. We then easily infer that $|u_1|^2\restr-\frac{\|u_1\|^2_{L^2(\Omega)}}{3|\Omega|}$ is annihilated by all smooth functions supported on $\Gamma_S$ and hence it must be zero a.e. on $\Gamma_S$. In the case that $\Om$ is stably convex, by definition the domain is stable under any local flow of a vector field, so the same argument as before yields that $|u_1|^2\restr=\frac{\|u_1\|^2_{L^2(\Omega)}}{3|\Omega|}$ a.e. on $\pd\Om$, thus completing the proof of the lemma.
\end{proof}


We finish this section showing that if an optimal convex domain is regular enough, it cannot be stably convex. In the statement we use the notion of domains in a Sobolev class $W^{2,p}$ as introduced in~\cite{MP14}. This class is natural in the sense that $W^{2,\infty}=C^{1,1}$, and it is proved~\cite{Amrou} that any curl eigenfield on a $C^{1,1}$ domain is continuous up to the boundary. Working with domains of class $W^{2,p}$, which are $C^1$ provided that $p>3$, is a natural setting to improve the $C^{1,1}$ regularity assumption. In particular, we do not know if the following holds for arbitrary $C^1$ domains:

\begin{proposition}\label{MT3}
Let $\Omega\subset \mathbb{R}^3$ be an optimal convex domain in the Sobolev class $W^{2,p}$ for some $p>3$. Then $\Omega$ is not stably convex. In particular, if $\Om$ is of class $C^{1,1}$, it is not strongly convex.
\end{proposition}
\begin{proof}
Assume that $\Om$ is stably convex and $u_1$ is a first curl eigenfield. We recall that a $C^{1,1}$ strongly convex domain is also stably convex, cf. Proposition~\ref{L.stable}. According to Lemma~\ref{L.norm}, $|u_1|^2=c$ a.e. on $\pd\Om$ for some constant $c>0$. We claim that $u_1\restr\in C^{0,\alpha}(\pd\Om)$ for some $\alpha>0$. Indeed, by Sobolev embedding we know that $\pd\Om$ is in the Sobolev class $W^{2-\frac1q,q}$, where $q:=\text{min}\{p,4\}$. Now, since $q>3$, it follows from~\cite{MP14} that $u_1$ is in $W^{1,q}(\overline{\Om})\subset C^{0,\alpha}(\overline\Om)$ for some $\alpha>0$, and so the boundary restriction $u_1|_{\pd\Om}$ is continuous. Since $\pd\Om$ is homeomorphic to a sphere, any continuous vector field must vanish at some point, which is a contradiction with the fact that $|u_1|^2=c>0$ everywhere. The proposition then follows.
\end{proof}


\section{Existence of uniform H\"{o}lder optimal domains}\label{SS.Holder}

Our goal in this section is to prove the existence of optimal domains within the class of bounded domains, of fixed volume and contained in a fixed bounded domain $D\subset\RR^3$, which are uniformly H\"older in the following sense.

\begin{definition} \label{D.UniHoel}
Given any constants $R>0$, $0<c_0<L$, $c_1>0$, $0<\alpha\leq 1$ and an integer $k\geq 2$, we define the set of uniform H\"{o}lder functions as
	\begin{align*}
C^{k,\alpha}(L,c_0,c_1):=&\{f\in C^{k,\alpha}(\overline{B_R}): \|f\|_{C^{k,\alpha}(\overline{B_R})}\leq L,\text{ } f^{-1}(0)\cap B_R \text{ is nonempty}
\\&\text{ and } |\nabla f|\geq c_0\text{ on } f^{-1}(t) \text{ for all }t\in(-c_1,c_1)\}\,.
\end{align*}
The class of uniformly H\"{o}lder domains $\mathcal{D}^{k,\alpha}(L,c_0,c_1)$ consists of domains $\Om\subset D$ ($D$ is fixed for all domains, and we take $R$ large enough so that $\overline{D}\subset B_R$), with a defining function $f\in C^{k,\alpha}(L,c_0,c_1)$. Therefore, $\Om=\{x\in B_R: f(x)<0\}$ and $B_R\backslash\overline{\Om}=\{x\in B_R: f(x)>0\}$, so that $\pd\Om=f^{-1}(0)\cap B_R=f^{-1}(0)\cap \overline{D}$.
\end{definition}

\begin{remark}
Any $C^{k,\alpha}$ bounded domain $\Om\subset D$ is in $\mathcal{D}^{k,\alpha}(L,c_0,c_1)$ if the constants $L,c_0^{-1},c_1^{-1}$ are large enough. We also observe that the principal curvatures of the boundary of any domain in $\mathcal{D}^{k,\alpha}(L,c_0,c_1)$ are uniformly bounded, so we are considering domains of uniformly bounded geometry.
\end{remark}

We recall that, for a certain~$V\in(0,|D|)$, we say that the constants are {\em admissible} when there are domains in $\mathcal{D}^{k,\alpha}(L,c_0,c_1)$ of volume~$V$.

\begin{theorem} \label{T.UniHoel}
Fix $k\geq2$, $\al>0$, $V\in(0,|D|)$ and admissible  constants $L$, $c_0$, $c_1$ as in Definition~\ref{D.UniHoel}. There exists a domain $\Omega$ of volume $V$ within the class $\mathcal{D}^{k,\alpha}(L,c_0,c_1)$ that is optimal for the first positive curl eigenvalue, that is, such that
	\[
	\mu_1(\Omega)=\inf_{\widetilde{\Omega}\in \mathcal{D}^{k,\alpha}(L,c_0,c_1)\text{, }|\widetilde{\Omega}|=V}\mu_1(\widetilde{\Omega})\,.
	\]
\end{theorem}

To prove this theorem we first establish the following compactness lemma. We recall that $d_H$ denotes the Hausdorff distance between compact sets and, as usual~\cite{Hen06}, we define $d_H(\Om_1,\Om_2):=d_H(\overline{B_R}\backslash\Om_1,\overline{B_R}\backslash\Om_2)$ for any two domains in $\mathcal{D}^{k,\alpha}(L,c_0,c_1)$.

\begin{lemma}\label{L.HC}
The space $\left(\mathcal{D}^{k,\alpha}(L,c_0,c_1),d_H\right)$ is a compact metric space. Moreover, the function $$\operatorname{Vol}:\left(\mathcal{D}^{k,\alpha}(L,c_0,c_1),d_H\right)\rightarrow \mathbb{R}$$
that assigns to a domain $\Omega$ its volume, $\Vol(\Om):=|\Om|$, is continuous.
\end{lemma}

\begin{proof}
Obviously all the elements of $\mathcal{D}^{k,\alpha}(L,c_0,c_1)$ are bounded and nonempty, so $d_H$ defines a metric on this space. Let $\{\Om_n\}_n\subset \mathcal{D}^{k,\alpha}(L,c_0,c_1)$ be any given sequence and fix defining functions $f_n\in C^{k,\alpha}(L,c_0,c_1)$ as described in Definition~\ref{D.UniHoel}. Due to the uniform H\"{o}lder bound we may assume that the sequence $\{f_n\}$ converges to some $f\in C^{k,\alpha}(\overline{B_R})$ in the $C^k$-norm. It is easy to check that the limiting function $f$ satisfies $\|f\|_{C^{k,\alpha}(\overline{B_R})}\leq L$ and $|\nabla f|\geq c_0$ on any level set $f^{-1}(t)$ for $t\in(-c_1,c_1)$. To see that $f\in C^{k,\alpha}(L,c_0,c_1)$ it remains to check that $f^{-1}(0)\cap B_R$ is nonempty. But, since $f_{n}^{-1}(0)\cap B_R\subset \overline{D}$ is nonempty for all $n$, it is obvious that $f^{-1}(0)\cap B_R\neq \emptyset$. Now, for any $N>0$ there is a large enough $N'$ such that for all $n\geq N'$
\[
\|f-f_n\|_{C^k(B_R)}<\frac{1}{N}\,.
\]
Since $k\geq 2$, Thom's isotopy theorem~\cite[Section~3]{Advances} then implies that $f^{-1}(0)\cap B_R$ is diffeomorphic to $f_n^{-1}(0)\cap B_R$, the diffeomorphism being close to the identity. More precisely, for any $N>0$, there is $N'>0$ such that for all $n\geq N'$ there is a $C^{k-1}$ diffeomorphism $\Phi_n:\RR^3\to\RR^3$ with
\begin{equation}\label{eq.diffeo1}
\|\Phi_n-id\|_{C^{k-1}(\RR^3)}<\frac{C}{N}\,,
\end{equation}
for some $n$-independent constant $C$, and
\begin{equation}\label{eq.diffeo}
f^{-1}(0)\cap B_R=\Phi_n(f_n^{-1}(0)\cap B_R)\,.
\end{equation}
In particular, $f^{-1}(0)\cap B_R$ is contained in $\overline{D}$. Moreover, $\Phi_n$ can be taken to be different from the identity only on a neighborhood of $\overline D$. Since $\{f_n\}$ are defining functions, it is obvious that $f$ is a defining function for the domain $\Om$ bounded by the surface $f^{-1}(0)$, and clearly $\Om\subset D$. It then follows from Equation~\eqref{eq.diffeo} that
\[
\Om=\Phi_n(\Om_n)
\]
for all $n$ large enough. Accordingly,
$$d_H(\Om,\Om_n)=d_H(\Phi_n(\Om_n),\Om_n)\to 0$$
when $n\to \infty$, where we have used the estimate~\eqref{eq.diffeo1} with $N\to\infty$. Since $\Om\subset \mathcal{D}^{k,\alpha}(L,c_0,c_1)$, this completes the proof of the first part of the lemma.

To prove the continuity of the volume function, we just observe that
\[
\partial\Om = f^{-1}(0)\cap B_R=\Phi_n(f_n^{-1}(0)\cap B_R)=\Phi_n(\pd\Om_n)
\]
for a diffeomorphism $\Phi_n$ that is as close to the identity (in the $C^{k-1}$ norm) as desired, provided that $n$ is large enough. This immediately implies that $|\Om_n| \to|\Om|$ as we wanted to show.
\end{proof}

Next we claim that the function $\mu_1:\left(\mathcal{D}^{k,\alpha}(L,c_0,c_1),d_H\right)\rightarrow\mathbb{R}$ that assigns to each domain $\Om$ its first (positive) curl eigenvalue $\mu_1(\Om)$ is continuous. Since Lemmas~\ref{L.Scaling} and~\ref{L.MP} hold for domains in $\mathcal{D}^{k,\alpha}(L,c_0,c_1)$, the proof of this result is essentially the same as in Proposition~\ref{C.c}. Indeed, given any sequence $\{\Om_n\}_n\subset \mathcal{D}^{k,\alpha}(L,c_0,c_1)$ converging to some $\Om\in \mathcal{D}^{k,\alpha}(L,c_0,c_1)$ in the Hausdorff metric, we can assume, as in the proof of Lemma~\ref{L.HC}, that the defining functions $f_n\in C^{k,\alpha}(L,c_0,c_1)$ of $\Om_n$ converge in the $C^k$-norm to a defining function $f\in C^{k,\alpha}(L,c_0,c_1)$ of $\Om$ (by the uniqueness of the Hausdorff limit). Now, the only observation to take into account is that the existence of the diffeomorphism $\Phi_n$ in Equation~\eqref{eq.diffeo}, which is $C^{k-1}$-close to the identity, implies, as in the convex setting, that there are sequences $\{t_n\}_n,\{s_n\}_n$ of real numbers, both converging to $1$, such that
\[
t_n\Omega_n\subseteq \Omega\subseteq s_n\Omega_n\,.
\]
The continuity of $\mu_1:\left(\mathcal{D}^{k,\alpha}(L,c_0,c_1),d_H\right)\rightarrow\mathbb{R}$ then follows from the proof of Proposition~\ref{C.c}.

Finally, to complete the proof of Theorem~\ref{T.UniHoel}, let us consider a minimizing sequence $\{\Om_n\}_n\subset \mathcal{D}^{k,\alpha}(L,c_0,c_1)$ of fixed volume $V$ for the first curl eigenvalue $\mu_1$. By Lemma~\ref{L.HC} we can extract a subsequence converging to some $\Om\in \mathcal{D}^{k,\alpha}(L,c_0,c_1)$ with $$|\Om|=\lim_{n\rightarrow\infty}|\Om_n|=V\,.$$
Additionally, the continuity of $\mu_1$ in the Hausdorff metric argued above yields that $$\mu_1(\Om)=\lim_{n\rightarrow\infty}\mu_1(\Om_n)=\inf_{\widetilde{\Om}\in \mathcal{D}^{k,\alpha}(L,c_0,c_1)\text{, }|\widetilde{\Om}|=V}\mu_1(\widetilde{\Om})\,,$$
thus implying that $\Om$ is an optimal domain, as we wanted to show.

\section*{Acknowledgements}

The authors are grateful to Rainer Picard for providing them with a copy of Ref.~\cite{Picard}. Wadim Gerner would like to thank Kristin L\"{u}ke for a concise introduction into the convex optimization of the Dirichlet Laplacian and for pointing out Ref.~\cite{CoFi10}. This work has received funding from the European Research Council (ERC) under the European Union's Horizon 2020 research and innovation programme through the grant agreement~862342 (A.E.). It is partially supported by the grants CEX2019-000904-S, RED2018-102650-T, and PID2019-106715GB GB-C21 (D.P.-S.) funded by MCIN/AEI/10.13039/501100011033.

\appendix

\section{Definition and properties of the $\Gamma$-convergence for the curl operator}\label{app}

The notion of $\Gamma$-convergence was introduced by de Giorgi to study the existence of solutions and their regularity for variational problems. In the context of optimal domains for spectral problems, this has been particularly useful to prove the existence of $C^1$ optimal convex domains for the Dirichlet eigenvalues of the Laplacian~\cite{Bucur}. In this Appendix we introduce a suitable notion of $\Gamma$-distance for vectorial boundary problems involving the curl operator, which enjoys some nice continuity properties (cf. Proposition~\ref{T.GaEig}). Although we have not been able to use these ideas to study the regularity of the optimal convex domains obtained in Theorem~\ref{T1}, we include a summary of our results because we think they are of independent interest.

In what follows we fix a smooth bounded domain $D\subset \mathbb{R}^3$. If $\Om\subseteq D$ is a Lipschitz domain and $w$ is a vector field in $\mathcal{K}(D)$, let us consider the unique solution of the boundary value problem:
	\begin{equation}
		\label{E.bvp}
		\operatorname{curl} v=\pi_{\Om}(w)\text{ in }\Om\text{,  }v\in \mathcal{K}(\Om)\,,
	\end{equation}
where $\pi_{\Om}(w)$ denotes the $L^2$-orthogonal projection of $w|_{\Om}$ into $\mathcal{K}(\Om)$. In terms of the Biot-Savart operator
\[
\operatorname{BS}(u)(x):=\frac{1}{4\pi}\int_{\Om}\frac{u(y)\times (x-y)}{|x-y|^3}dy\,,
\]
we notice that if $u\in \mathcal{K}(\Om)$, then $\operatorname{BS}(u)\in H^1(\Om)$ satisfies $\operatorname{div}(\operatorname{BS}(u))=0$ and $\operatorname{curl}(\operatorname{BS}(u))=u$. It is then easy to check that the solution $v$ to Equation~\eqref{E.bvp} is given by
\begin{equation}\label{E.mbs}
v=	\pi_{\Om}\left(\operatorname{BS}(\pi_{\Om}(w))\right).
\end{equation}
Finally, we define the operator $\widetilde{\operatorname{BS}}_{\Om}:\cK(D)\to \cK(D)$ as
\begin{equation*}
\widetilde{\operatorname{BS}}_{\Om}(w):= \begin{cases}
	\pi_{\Om}\left(\operatorname{BS}(\pi_{\Om}(w))\right)&\text{on }\Omega\,,\\
	0&\text{on }D\backslash \Om\,.
\end{cases}
\end{equation*}
It is standard to check that $\widetilde{\operatorname{BS}}_{\Om}$ is a compact self-adjoint operator whose eigenvalues are given by $\{\mu_k^{-1}(\Om)\}_{k\in \ZZ_0}$, where $\mu_k$ are the eigenvalues of the $\curl$ operator in $\Om$.

\begin{definition}\label{D.Gamma}
Given two Lipschitz domains $\Om,\Om'\subseteq D$, we define their {\em $\Gamma$-distance} as
\[
d_{\Gamma}(\Om,\Om'):=\sup_{w\in \mathcal{K}(D)\text{, }\|w\|_{L^2(D)}=1}\|\widetilde{\operatorname{BS}}_{\Om}(w)-\widetilde{\operatorname{BS}}_{\Om'}(w)\|_{L^2(D)}\,.
\]
The fact that $d_{\Gamma}$ defines a metric on the collection of Lipschitz domains contained in $D$ is elementary.
\end{definition}

The following result shows that the eigenvalues of the curl operator are Lipchitz continuous with respect to the $\Gamma$-distance. An analogous result was crucial in~\cite{Bucur} to prove the $C^1$ regularity of the optimal convex domains for the Dirichlet eigenvalues of the Laplacian.

\begin{proposition} \label{T.GaEig}
Let us denote by $\{\la_k(\Om)\}_{k=1}^\infty$ the collection of the absolute values $\{|\mu_j(\Om)|\}_{j\in\ZZ_0}$ of the eigenvalues of the curl operator on $\Om$ (in increasing order). If $\{\Om_n\}_n\subset D$ are Lipschitz domains which converge to a Lipschitz domain $\Om\subset D$ with respect to $d_\Gamma$, then for all $k\in \mathbb{N}$ we have $\lambda_k(\Om_n)\rightarrow \lambda_k(\Om)$ as $n\rightarrow\infty$. More precisely, we have the estimate
	\[
	\left|\frac{1}{\lambda_k(\Om)}-\frac{1}{\lambda_k(\Om^\prime)}\right|\leq d_{\Gamma}(\Om,\Om^\prime)
	\]
	for any two Lipschitz domains $\Om,\Om^\prime\subseteq D$ and all $k\in \mathbb{N}$.
\end{proposition}

\begin{proof}
It is immediate from Definition~\ref{D.Gamma} that $$d_{\Gamma}(\Om,\Om')=\|\widetilde{\operatorname{BS}}_{\Om}-\widetilde{\operatorname{BS}}_{\Om^\prime}\|_{L^2(D)}\,,$$
where $\|\cdot\|_{L^2(D)}$ denotes the operator norm. Since $\widetilde{\operatorname{BS}}_{\Om}$ and $\widetilde{\operatorname{BS}}_{\Om'}$ are compact self-adjoint operators on a Hilbert space, denoting by $\nu_k(\Om),\nu_k(\Om')$, $k\in \mathbb{N}$, the absolute values of the eigenvalues of $\widetilde{\operatorname{BS}}_{\Om}$ and $\widetilde{\operatorname{BS}}_{\Om'}$, respectively, ordered by size starting with the largest and counting multiplicities, then~\cite[Theorem 2.3.1]{Hen06}
\[
\|\widetilde{\operatorname{BS}}_{\Om}-\widetilde{\operatorname{BS}}_{\Om^\prime}\|_{L^2(D)}\geq \left|\nu_k(\Om)-\nu_k(\Om')\right|
\]
for all $k\in\NN$. Since $\nu_k(\Om)=\frac{1}{\la_k(\Om)}$ as mentioned above, the proposition follows.

\end{proof}

\bibliographystyle{amsplain}

\end{document}